\documentclass[runningheads]{llncs}
\usepackage[T1]{fontenc}

\usepackage{graphicx}
\usepackage{thmtools}
\usepackage{thm-restate}
\usepackage{amsmath}
\usepackage{verbatim}

\usepackage{hyperref}

\usepackage{tikz}
\usetikzlibrary{arrows, quotes}
\usetikzlibrary{arrows.meta}

\begin{document}

\title{Antimagic labellings of $(k,2)$-bipartite biregular graphs}
\author{Grégoire Beaudoire\inst{1}\orcidID{0000-0002-3944-7737} \and
Cédric Bentz\inst{1}\orcidID{0009-0009-8597-1669} \and
Christophe Picouleau\inst{1}\orcidID{0000-0001-8092-1923}}
\authorrunning{G. Beaudoire et al.}

\institute{{CNAM, Paris}}
\maketitle 

\begin{abstract}

An antimagic labelling of a graph $G = (V,E)$ is a bijection from $E$ to $\{1,2, \ldots, |E|\}$, such that all vertex-sums are pairwise distinct, where the vertex-sum of each vertex is the sum of labels over edges incident to this vertex. A graph is said to be antimagic if it has an antimagic labelling. Recently, it has been proven that $(s,t)$-bipartite biregular graphs are antimagic if $|s - t| \geq 2$ and $s$ or $t$ is odd. In this paper, we extend this result to connected $(k,2)$-bipartite biregular graphs for $k \geq 4$ even, and to $(k,2)$-bipartite biregular graphs for $k \geq 3$ odd.

\keywords{Antimagic labelling  \and Bipartite graphs \and Biregular graphs}
\end{abstract}

\section{Introduction and definitions}

In this paper, we only consider finite, simple, and undirected graphs. We refer to \cite{west} for undefined terminology.

Let $G=(V,E)$ be a graph with $|V| = n$ and $|E| = m$. We denote by $d_G(v)$ the degree of a vertex $v\in V$. If $d_G(v)=k$ for every vertex $v\in V$, then $G$ is \emph{$k$-regular}.

If $G=(V,E)$ is a bipartite graph, then $V = X \cup Y$, with $E \subseteq X \times Y$; in this case, we write $G=(V = X \cup Y,E)$. Such a graph is a \emph{biregular bipartite graph} iff there exist $k$ and $k'$ such that $d_G(x)=k$ for each $x \in X$ and $d_G(y)=k'$ for each $y \in Y$. We also call such a graph a \emph{$(k,k')$-bipartite graph}. Hence, in a $(k,2)$-bipartite graph, we have $d_G(x)=k$ for each $x \in X$ and $d_G(y)=2$ for each $y \in Y$.

The subgraph of $G=(V,E)$ \emph{induced by} a set $S\subseteq V$ is denoted by $G[S]$, and defined as $G[S]=(S,F)$ where $F=\{xy\in E : x,y\in S\}$.

We also denote by $K_p$ the complete graph with $\vert V \vert = p$, and by $K_{p,q}$ the complete bipartite graph with $\vert X \vert = p$ and $\vert Y \vert = q$.

Given a graph $G = (V,E)$, let $f : E \rightarrow  \{1, 2, \ldots, m\}$ be a bijective labelling of the edges of $G$. For each vertex $u \in V$, we will denote by $\sigma(u) = \sum\limits_{v \in V : uv \in E} f(u,v)$ the sum of the labels over edges incident to $u$. If all the values $\sigma(u)$ are pairwise distinct, then $f$ is called an \emph{antimagic labelling} of $G$. If $G$ admits at least one antimagic labelling, then $G$ is said to be \emph{antimagic}.

Antimagic labelling was originally introduced by Hartsfield and Ringel in $1990$ \cite{hartsfield1990}, where they introduced the following conjecture:

\begin{conjecture}\label{conj}
    Every connected graph other than $K_2$ is antimagic.
\end{conjecture}

The topic is the focus of a chapter of 12 pages in the dynamic survey on graph labelling by J. Gallian \cite{survey}.\\ 

Our paper mainly focuses on $(k,2)$-bipartite  graphs for $k \geq 3$. It is worthwhile to note that Conjecture \ref{conj} is not true for non-connected graphs, even if they do not have $K_2$ as a connected component. Such an example is the graph $2 P_3$, which corresponds to two copies of the antimagic graph $P_3$, and which is also a (non-connected) $(1,2)$-bipartite graph.

However, Conjecture \ref{conj} was proved to be true in several special cases, in particular in dense graphs \cite{alon2004}, and in some subclasses of trees (and hence also of connected bipartite graphs) \cite{deng2019,kaplan09,liang14DM}. Moreover, some basic results can easily be proved (see \cite{survey}): for instance, if $G$ is a cycle (a connected $2$-regular graph), a collection of cycles (a non-connected $2$-regular graph), or a path distinct from $K_2$, then $G$ is antimagic.

We now survey the most recent results concerning antimagic labellings of $k$-regular and  $(k,2)$-bipartite graphs. In $2014$, \cite{liang14JGT} proved that cubic graphs (i.e., $3$-regular graphs) are antimagic. This result was later extended to regular graphs with odd degree in \cite{cranston2015}, and then independently by \cite{Berczi2015} and \cite{chang2015} to all regular graphs in 2015:

\begin{theorem}[\cite{Berczi2015} and \cite{chang2015}]
    Regular graphs are antimagic.
\end{theorem}

Remark that, for $k$-regular graphs, $k \geq 2$, the ``connected'' property does not matter, as labelling of connected components can easily be translated without incurring conflicts that one would expect to happen in a (non-regular) graph with several connected components.

The most recent result concerning biregular bipartite graphs, proven by Yu in $2023$ \cite{Yu2023}, is the following:

\begin{theorem}[Theorem 1.5 in \cite{Yu2023}]
    Let $G = (V = X \cup Y, E)$ be an $(s,t)$-bipartite  graph. If $s \geq t+2$ and one of $s$ or $t$ is odd, then $G$ is antimagic.
\end{theorem}

The cases where $s = t + 1$ or both $s$ and $t$ are even are still open.

At about the same time, some results related to this question were independently proved in \cite{deng2022}. Namely, the following results were proved:

\begin{theorem}[Theorem 2 in \cite{deng2022}]
    Let $G$ be a connected $k$-regular graph. Then the graph $G^{\prime}$, obtained by subdividing every edge in $G$ exactly once, is antimagic.
\end{theorem}

However, it should be noticed that the authors claim in \cite{deng2022} that every $(k,2)$-bipartite graph can be obtained by subdividing exactly once every edge (i.e., by replacing each edge $uv$ by two new edges $uw$ and $wv$, where $w$ is a new vertex) of a $k$-regular graph. Hence, they conclude that the above result also applies to every connected $(3,2)$-bipartite graph. 

Unfortunately, in the case of $k$-regular simple graphs, this is not true in general. Indeed, it can easily be seen that a $(k,2)$-bipartite graph with $k \geq 2$ can be obtained by subdividing every edge of a $k$-regular simple graph exactly once if and only if it does not contain $C_4$ as a (necessarily induced) subgraph. Besides, it does not seem that the labelling process described in \cite{Berczi2015} for simple regular graphs can be easily adapted to the case of non-simple regular graphs.

More recently, some additional results were proven over the subdivisions of regular graphs with various specific conditions in \cite{wei2025}.

In the present paper, we manage to settle the case of $(3,2)$-bipartite graphs, not necessarily connected, with a proof easily adaptable to $(k,2)$-bipartite graphs for $k \geq 3$ odd. Note that the case $k > 3$ with $k$ odd was already proved, with a different construction, in \cite{Yu2023}. We show the following results:

\begin{restatable}{theorem}{Connected}
\label{Connected}
    Every connected $(k,2)$-bipartite graph, with $k \geq 3$ odd, is antimagic.
\end{restatable}

\begin{restatable}{theorem}{Unconnected}
    \label{Unconnected}
    Every $(k,2)$-bipartite graph, with $k \geq 3$ odd, is antimagic.
\end{restatable}

We also provide another labelling process that allows us to extend the results for $(k,2)$ bipartite graphs, with $k$ even:

\begin{restatable}{theorem}{ConnectedkEven}
    \label{ConnectedkEven}
    Every connected $(k,2)$-bipartite graph, with $k \geq 4$ even, is antimagic.
\end{restatable}

The article is organized as follows. In Section $2$, we give the notations and the basic notions that we will use in our labellings in this article. In Section $3$, we explain the labelling algorithm for connected $(k,2)$-bipartite graphs, with $k \geq 3$ odd. Its correctness is then proven in Section $4$. In Section $5$, we show how to generalize it to non-connected graphs. In Section $6$, we show how to obtain an antimagic labelling for connected $(k,2)$-bipartite graphs, with $k \geq 4$ even.

\section{Notations and partitions of the graph}
\label{Notations}

Let $G = (V = X\cup Y,E)$ be a $(k,2)$-bipartite graph. We assume for now that $G$ is connected. Assume that $G$ is rooted at an arbitrary vertex $r$ of degree $k$. We partition the vertex set of $G$ depending on the distance from $r$, in such a way that $V = V_0 \sqcup V_1 \sqcup V_2 \ldots \sqcup V_l$ for some $l$, where $V_0 = \{r\}$, and where, throughout the paper, the symbol $\sqcup$ will be used to link together the parts of a given partition. This means that, for each vertex $x \in V$, $x \in V_i$ for some $i$ if and only if the distance from $r$ to $x$ is $i$ in $G$. Note that this defines a partition of $V$, since $G$ is assumed to be connected. Also note that this partition is such that, for any $x \in V_i$ with $i$ even, we have $d_G(x) = k$, and, for any $y \in V_j$ with $j$ odd, we have $d_G(y) = 2$. Moreover, there is no edge between any two vertices belonging to the same subset $V_i$.

For any $i$ and any given vertex $x \in V_i$, we will call \emph{predecessor of $x$} any vertex $y \in V_{i-1}$ such that $xy \in E$, and \emph{successor of $x$} any vertex $y \in V_{i+1}$ such that $xy \in E$ (assuming such a vertex $y$ exist).

One key idea in each of our labellings is that, at the end of the process, for each $u \in V$, $\sigma(u)$ is odd if $d_G(u) = 2$, and $\sigma(u)$ is even if $d_G(u) = k$ and $u \neq r$. This ensures that $\sigma(u) \neq \sigma(v)$ for each $u,v \in V \setminus\{r\}$ such that $d_G(u) = 2$ and $d_G(v) = k$.

We will label the edges of $G$ by starting from the farthest edges from $r$, and then progressing towards $r$. In order to do that, we will partition the edge set of $G$ into \emph{layers} $L_i$ (for $ 1 \leq i \leq p$, with $p = \lfloor \frac{l-1}{2} \rfloor + 1$). As shown in Figure \ref{fig2}, for each $i$, $L_i$ is the set of edges with an endpoint $v \in V_{2i-1}$.

\begin{figure}[!ht]
\centering
\begin{tikzpicture}[scale=0.5]
    \node at (0,0) (v1) {$\bullet$};
    \node at (3,0) (u1) {$\bullet$};
    \draw (-2,0) -- (v1.center) -- (u1.center);
    \draw[dashed] (5,-2) -- (u1.center) -- (5,2);
    \node at (0,-2) (v2) {$\bullet$};
    \node at (3,-4) (u2) {$\bullet$};
    \node at (0,-6) (v3) {$\bullet$};
    \draw (-2,-2) -- (v2.center) -- (u2.center) -- (v3.center) -- (-2,-6);
    \draw[dashed] (u2.center) -- (5,-4);
    \node at (-2,0) (w1) {$\bullet$};
    \node at (-2,-2) (w2) {$\bullet$};
    \node at (-2,-6) (w3) {$\bullet$};

    \node at (0,0.3) {$v_1$};
    \node at (3,0.3) {$u_1$};
    \node at (0,-1.7) {$v_2$};
    \node at (3,-3.7) {$u_2$};
    \node at (0,-5.7) {$v_3$};

    \draw[dashed] (-4,-1) -- (w1.center) -- (-4,1);
    \draw[dashed] (-4,-3) -- (w2.center) -- (-4,-2);
    \draw[dashed] (-4,-6) -- (w3.center) -- (-4,-5);

    \node at (-2,0.3) {$u_3$};
    \node at (-2,-1.7) {$u_4$};
    \node at (-2,-5.7) {$u_5$};

    \draw[red,thick] (0,-2.5) ellipse (0.8cm and 4cm);
    \node[red,thick] at (0,-7) {$V_{2i-1}$};

    \draw[blue,thick] (-2,-2.5) ellipse (0.8cm and 4cm);
    \node[blue, thick] at (-2,-7) {$V_{2i-2}$};

    \draw[blue,thick] (3,-1.9) ellipse (0.8cm and 2.5cm);
    \node[blue,thick] at (3,-5) {$V_{2i}$};

\end{tikzpicture}
\caption{A layer $L_i$ (full edges are in $L_i$, while dashed edges are not).}\label{fig2}
\end{figure}

We will first reserve an interval $[a_i; b_i]$ for each $L_i$, corresponding to the labels that will be used for the edges in $L_i$.

Note that, by definition, we have $E = L_1 \sqcup L_2 \sqcup \ldots \sqcup L_p$. We start by reserving the interval $[1; 2|V_{2p-1}|]$ for the labels of edges in $L_p$ (the farthest edges from $r$). We then reserve the next interval $[2|V_{2p-1}| + 1; 2|V_{2p-1}| + 2|V_{2p-3}|]$ for the labels of edges in $L_{p-1}$, and so on. Overall, for each $i \in \{0, \dots,  p-2\}$, we reserve the interval  $[2|V_{2p-1}| + 2|V_{2p-3}| + \ldots + 2|V_{2(p-i) - 1}| + 1; 2|V_{2p-1}| + 2|V_{2p-3}| + \ldots + 2|V_{2(p-i-1)-1}|]$ for the labels of the edges in $L_{p-i-1}$.

Consider now a given layer $L_i$ for some $i$, with its reserved interval $[a_i; b_i]$ (with $b_i = a_i + 2|V_{2i-1}| - 1$). 

We can first notice that, by the construction of the intervals, $a_i$ is odd and $b_i$ is even (since the number of edges to label in $L_i$ is $2|V_{2i-1}|$, an even number).

For any $i \geq 1$, we now categorize the vertices of $V_{2i}$ into $k$ different types (see Figure \ref{fig3} for an illustration with $k = 3$): the vertices in $V_{2i}$ will be of type $j \geq 1$ iff they have $j$ incident edges towards predecessors and $k-j$ incident edges towards successors. For any $i,j$, we will denote by $t_i^j$ the number of vertices in $V_{2i}$ of type $j$.

\begin{figure}[!ht]
\centering
\begin{tikzpicture}
    \node at (0,0) (v) {$\bullet$};
    \node at (1,0) (u) {$\bullet$};
    \draw (v.center) -- (-1,0);
    \draw (u.center) -- (v.center);
    \draw (u.center) -- (2,0.6);
    \draw (u.center) -- (2,-0.6);
    \node at (1,-0.3) {$u$};
    \node at (2,0.6) {$\bullet$};
    \node at (2,-0.6) {$\bullet$};
    \node at (-1,0) {$\bullet$};

    \node at (4,1) (v1) {$\bullet$};
    \node at (4,-1) (v2) {$\bullet$};
    \node at (5,0) (u) {$\bullet$};
    \draw (v1.center) -- (u.center) -- (v2.center);
    \draw (v1.center) -- (3,1);
    \draw (v2.center) -- (3,-1);
    \draw (u.center) -- (6,0);
    \node at (5,0.3) {$u$};
    \node at (3,1) {$\bullet$};
    \node at (3,-1) {$\bullet$};
    \node at (6,0) {$\bullet$};

    \node at (9,0) (u) {$\bullet$};
    \node at (8,1) (v1) {$\bullet$};
    \node at (8,0) (v2) {$\bullet$};
    \node at (8,-1) (v3) {$\bullet$};
    \draw (v1.center) -- (u.center) -- (v2.center);
    \draw (u.center) -- (v3.center);
    \draw (v1.center) -- (7,1);
    \draw (v2.center) -- (7,0);
    \draw (v3.center) -- (7,-1);
    \node at (9,0.3) {$u$};
    \node at (7,1) {$\bullet$};
    \node at (7,0) {$\bullet$};
    \node at (7,-1) {$\bullet$};

    \draw (0,0) ellipse (0.3cm and 1cm);
    \node at (0,-1.3) {$V_{2i-1}$};
    \draw (1,0) ellipse (0.3cm and 1cm);
    \node at (2,-1.3) {$V_{2i+1}$};
    \draw (2,0) ellipse (0.3cm and 1cm);
    \node at (1,-1.3) {$V_{2i}$};

    \draw (4,0) ellipse (0.3cm and 2cm);
    \node at (4,-2.3) {$V_{2i-1}$};
    \draw (5,0) ellipse (0.3cm and 1cm);
    \node at (5,-1.3) {$V_{2i}$};
    \draw (6,0) ellipse (0.3cm and 1cm);
    \node at (6,-1.3) {$V_{2i+1}$};

    \draw (8,0) ellipse (0.3cm and 2cm);
    \node at (8,-2.3) {$V_{2i-1}$};
    \draw (9,0) ellipse (0.3cm and 1cm);
    \node at (9,-1.3) {$V_{2i}$};

    \node at (1,-3) {Type $1$};
    \node at (5,-3) {Type $2$};
    \node at (8,-3) {Type $3$};
\end{tikzpicture}

\caption{The different types for a vertex $u \in V_{2i}$, when $ k = 3$.}\label{fig3}
\end{figure}

Note that these $k$ types cover all vertices of $V_{2i}$, for each $i \geq 1$, since $G$ is connected. Similarly, we can categorize the vertices of $V_{2i-1}$ for each $i \geq 1$ into two types:

\begin{itemize}
    \item Type $1$: vertices that have exactly $1$ incident edge towards a predecessor and $1$ incident edge towards a successor.

    \item Type $2$: vertices that have $2$ incident edges towards predecessors.
\end{itemize}

Finally, at any time during the labelling process, for any $u \in V$, $\sigma'(u)$ will denote the partial sum of the labels over edges incident to $u$ that have already been labelled. When every edge incident to $u$ has been labelled, $\sigma'(u) = \sigma(u)$.

\section{Labelling process for connected $(k,2)$-bipartite graphs, with $k \geq 3$ odd}
\label{Labellingcon32}

In this section, we will describe the labelling process for each layer $L_i$. We assume that $G$ is a connected $(k,2)$-bipartite graph, with $k \geq 3$ odd, different from the complete bipartite graph $K_{k,2}$. Recall that we are successively labelling $L_p$, then $L_{p-1}$, and so on until $L_1$.

During the labelling process, when we label an edge incident to a vertex $x$ of degree $2$, with some label $\alpha$, then we immediately label the other edge incident to $x$ with $\alpha + 1$. This guarantees that, for each vertex $x$ of degree $2$, $\sigma(x)$ is odd at the end of the labelling.

For each layer $L_i$, we do the following three steps:

\begin{itemize}
    \item Step $1$: we build a set $F_i \subset L_i$, such that the two following properties are satisfied:

        \begin{itemize}
            \item For each vertex $x \in V_{2i-1}$ (i.e., $x$ has degree $2$), $x$ has either its two incident edges in $F_i$, or none of them.

            \item For each vertex $y \in V_{2i}$: if $y$ is of type $j$, then $y$ has $j-1$ incident edges in $E(V_{2i-1},V_{2i}) \cap F_i$. In other words, $y$ has all its incident edges in $E(V_{2i-1},V_{2i})$ that belong to $F_i$, except for one.
        \end{itemize}

    We will use a greedy algorithm to label $F_i$: we aim to label edges along maximal paths that we will explore step-by-step. Every time we go from a vertex in $V_{2i}$ to a vertex in $V_{2i-2}$, we will use the two smallest labels available, and conversely every time we go from a vertex in $V_{2i-2}$ to a vertex in $V_{2i}$, we will use the two largest labels available. In order to do so, we always have to label one edge in $E(V_{2i-2},V_{2i-1})$ as well as its only incident edge in $E(V_{2i-1},V_{2i})$, and we always assign the largest of the two labels to the edge in $E(V_{2i-2},V_{2i-1})$ and the smallest one to the edge in $E(V_{2i-1},V_{2i})$. This will allow us to obtain upper and lower bounds for the values of $\sigma^{\prime}(u)$ in order to guarantee that, at the end of the labelling, for every $u \in V_{2i-2}$ and $v \in V_{2i}$, $\sigma(u) > \sigma(v)$.

    Let us note $m_i = a_i$ and $M_i = a_i + |F_i| - 1$ respectively the smallest available label and the largest available label for the edges in $F_i$.

    We introduce the following definition:

    \begin{definition}
        Let $u,v,w \in G[V_{2i-2}\cup V_{2i-1} \cup V_{2i}]$. We will denote by $u \rightarrow v \rightarrow w$ a path of length $2$ from a vertex $u \in V_{2i-2}$ to a vertex $w \in V_{2i}$, and by $u \leftarrow v \leftarrow w$ a path of length $2$ from a vertex $w \in V_{2i}$ to a vertex $u \in V_{2i-2}$.
    \end{definition}

    We start the labelling by arbitrarily choosing a vertex $u \in V_{2i-2}$ with a neighbor $v \in V_{2i-1}$ such that $uv \in F_i$. If such a $u$ does not exist, then $F_i$ is empty and we can go to Step $2$. Let us denote by $w$ the neighbor of $v$ in $V_{2i}$. We label the edge $uv$ with $M_i$ and the edge $vw$ with $M_i - 1$.

    Formally, after $2j$ edges have been labelled in $F_i$ ($j \geq 1$, meaning this also applies after the initial step we just described), for some $u \in V_{2i-2},v \in V_{2i-1},w \in V_{2i}$:

    \begin{description}
        \item[Case $1$:] If we just labelled edges corresponding to $u \leftarrow v \leftarrow w$, then $j$ is even, and we used the labels $m_i + j - 1$ and $m_i + j - 2$, since we went from a vertex in $V_{2i}$ to a vertex in $V_{2i-2}$. Then:

        \begin{itemize}
            \item[a)] If there exists some $x \in V_{2i-1}$ such that the edge $ux$ is in $F_i$ and is not labelled yet, then we label the edges of $u \rightarrow x \rightarrow y$, where $y$ is the neighbor of $x$ in $V_{2i}$, with the labels $M_i - j$ and $M_i - j - 1$.

            \item[b)] If no such $x$ exists: we arbitrarily pick a vertex $u' \in V_{2i-2}$ such that $u'$ has a neighbor $x \in V_{2i-1}$ such that $u'x \in F_i$ and $u'x$ is not labelled yet, and we label the edges of $u' \rightarrow x \rightarrow y$, where $y$ is the neighbor of $x$ in $V_{2i}$, with the labels $M_i - j$ and $M_i - j - 1$. Note that, if no such $u'$ exists, this means that we have labelled every edge in $F_i$.
        \end{itemize}

        \item[Case $2$:] If we just labelled edges corresponding to $u \rightarrow v \rightarrow w$, then $j$ is odd, and we used the labels $M_i - j + 1$ and $M_i - j$. Then:

        \begin{itemize}
            \item[a)] If there exists some $x \in V_{2i-1}$ such that the edge $wx$ is in $F_i$ and is not labelled yet, then we label the edges of $y \leftarrow x \leftarrow w$, where $y$ is the neighbor of $x$ in $V_{2i-2}$, with the labels $m_i + j$ and $m_i + j - 1$.

            \item[b)] If no such $x$ exists: we arbitrarily pick a vertex $w' \in V_{2i}$ such that $w'$ has a neighbor $x$ such that $w'x \in F_i$ and $w'x$ is not labelled yet, and we label the edges of $y \leftarrow x \leftarrow w'$, where $y$ is the neighbor of $x$ in $V_{2i-2}$, with the labels $m_i+j$ and $m_i+j - 1$. If no such $w'$ exists, then we have labelled every edge in $F_i$.
        \end{itemize}
    \end{description}

    When we reach a vertex $u$ during this labelling, one of two things can happen: either we can extend the current path we are exploring by immediately labelling another edge incident to $u$, or we have to stop the current path on $u$ because it is not possible to extend it. We will explain in both cases how we obtain the upper and lower bounds we are looking for. Informally, we want to ``pair up'' the labels such that the sum of every pair will be either at most $m_i + M_i - 1$ (for $u \in V_{2i}$), or at least $m_i + M_i - 1$ (for $u \in V_{2i-2}$).
    
    If we reach a vertex $z \in V_{2i-2} \cup V_{2i}$ during the labelling of $F_i$, having labelled some edge $zz'$, and we can label another edge $zz''$ incident to $z$, we do so immediately. If we are in Case $1a$, the sum of the two labels (over $zz'$ and $zz''$) is equal to $(m_i + j - 1) + (M_i - j) = m_i + M_i - 1$, and in Case $2a$ the sum is also $(M_i - j) + (m_i + j - 1) = m_i + M_i - 1$. This allows us to keep pairing up the labels over the edges incident to a given vertex.
    
    If we reach a vertex $w \in V_{2i}$, and it is not possible to label another edge incident to $w$, then we are in Case $2b$ and one edge incident to $w$ was just labelled with some $\alpha = M_i - j \leq M_i$. This can only happen once for each $w$, since, if we reach $w$ and it is possible to extend the current path, we do so without stopping. Then, at least one of these two cases is true:

    \begin{itemize}
        \item $w$ has an odd number of incident edges in $F_i$, meaning $w$ has at least one incident edge towards a successor (since $k \geq 3$ is odd). This edge was labelled earlier in $L_{i+1}$, with a label $ \beta < m_i$, and we can pair it with $\alpha$ to obtain $\alpha + \beta \leq m_i + M_i - 1$.

        \item $w$ was chosen earlier during the labelling of $F_i$ to be the new ``start'' of the labelling, with some $\beta = m_i + j'-1$, with $j' < j$. In this case, $\alpha + \beta \leq m_i + M_i - 1$.
    \end{itemize}

    Recall that $k-1$ is even, and once all the edges in $F_i$ have been labelled, every vertex in $V_{2i}$ has exactly $k-1$ edges labelled. Notice that we have, so far, paired up an even number of edges, and hence we have an even number left to pair up (since, again, $k$ is odd). The only labelled edges incident to a vertex $w \in V_{2i}$ that remain unpaired are the edges that were labelled in $L_{i+1}$ (with a label $\beta < m_i$), and the ones that were labelled if $w$ was repeatedly chosen to be the new ``start'' of a path during the labelling of $F_i$, with some labels $m_i + j_1 - 1, m_i + j_2 - 1, \ldots, m_i + j_q - 1$, for some $q$. In this case, since those labels are picked among the ``smallest'' labels during the labelling of $F_i$, their maximal value is $m_i + \frac{\vert F_i \vert}{2} - 1 = m_i + \frac{M_i - m_i - 1}{2}$. We can then arbitrarily pair up two of those labels together and obtain a sum with a value at most $m_i + M_i - 1$.

    Overall, once all the edges in $F_i$ have been labelled, we obtain that, for each $w \in V_{2i}$, $\sigma'(w) \leq (m_i + M_i - 1)\frac{k-1}{2}$.

    Similarly, if we reach a vertex $u \in V_{2i-2}$, and it is not possible to label another edge incident to $u$, then we are in Case $1b$, one edge incident to $u$ was just labelled with some $\alpha = m_i + j - 1 > m_i$. Again, this can only happen once for each $u$, and at least one of these two cases is true:

    \begin{itemize}
        \item $u$ has an odd number of incident edges in $F_i$, meaning (since $k \geq 3$ is odd) that $u$ has at least two incident edges (and at least one in $L_{i-1}$) that are not in $F_i$, and hence that will be labelled once all of $F_i$ is labelled. This means that there is an edge incident to $u$, that belongs either to $E(V_{2i-2},V_{2i-1})$ or to $L_{i-1}$ (if $u$ has at least two incident edges in $L_{i-1}$), that will be labelled with some label $ \beta > M_i$, and we can pair it with $\alpha$ to obtain $\alpha + \beta \geq m_i + M_i - 1$.

        \item $u$ was chosen earlier during the labelling of $F_i$ to be the new (or the original) ``start'' of the labelling, with some $\beta = M_i - j'$ such that $j' < j$, meaning $\alpha + \beta \geq m_i + M_i - 1$.
    \end{itemize}

    For a given vertex $u \in V_{2i-2}$, the labelled edges incident to $u$ that remain unpaired are the ones that were repeatedly chosen to be the new ``start'' of a path during the labelling of $F_i$, with some labels in $K_1 = \{M_i - j'_1, M_i - j'_2, \ldots, M_i - j'_{q'}\}$, for some $q'$. In this case, since those labels are picked among the ``largest'' labels available during the labelling of $F_i$, their minimal value is $m_i + \frac{M_i - m_i + 1}{2}$. Moreover, if $u$ had $k' < k-1$ incident edges labelled during the labelling of $F_i$, $u$ will have $k - 1 - k'$ incident edges labelled after the labelling of $F_i$ (thus excluding the last incident edge in $L_{i-1}$), with labels in some set $K_2$, all greater than $M_i$. We can arbitrarily pair the labels in $K_1 \cup K_2$ together to obtain a set of sums, each with value at least $m_i + M_i - 1$. Overall, once $k-1$ edges incident to $u$ have been labelled, $\sigma'(u) \geq (m_i + M_i - 1) \frac{k-1}{2}$.

    \item Step $2$: there is now exactly one unlabelled edge incident to every vertex in $V_{2i}$. Notice that we have labelled an even number of edges during Step $1$ since every vertex in $V_{2i-1}$ has either $0$ or $2$ incident edges in $F_i$, meaning that the smallest available label $a_i + |F_i|$ for this step is odd, since $a_i$ is odd.

    We sort the $h$ vertices in $V_{2i}$ so that $\sigma'(u_1) < \sigma'(u_2) < \ldots < \sigma'(u_h)$, and iterate from $u_1$ to $u_h$. Let $u_jv$ be the last unlabelled edge incident to $u_j$ for some $j$, and $w$ the neighbor of $v$ in $V_{2i-2}$. Let $\alpha$ and $\alpha + 1$ be the two smallest available labels. Necessarily, $\alpha$ is odd. If $\sigma'(u_j)$ is odd, then we label $u_jv$ with $\alpha$ and $vw$ with $\alpha + 1$, otherwise we label $u_jv$ with $\alpha + 1$ and $vw$ with $\alpha$. This way, we guarantee that, at the end of the process, $\sigma(u_j)$ is even.

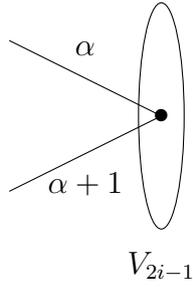
    \item Step $3$: we label the two edges incident to each vertex of type $2$ in $V_{2i-1}$ (i.e. to each vertex with $2$ predecessors and no successors) with consecutive labels $\alpha$ and $\alpha + 1$, starting from the smallest available one.
    
\end{itemize}

We will prove in the next section that, at the end of this process, the labelling obtained is antimagic for $G$, provided that $G$ is different from $K_{k,2}$. 

\section{Proof of Theorem \ref{Connected}}
\label{Claims}

We will use the following lemma to suppose that $G \neq K_{k,2}$ in the following (the proof of this lemma can be found in the Appendix).

\begin{restatable}{lemma}{lemmabic}
\label{lemmabic}
    For each $k \geq 3$, the complete bipartite graph $K_{k,2}$ is antimagic.
\end{restatable}

We now aim to prove our main result:

\Connected*

\begin{proof}
    Let $G=(V = X \cup Y, E)$ be such a graph. We label each layer $L_i$ of $G$, for $i$ from $p$ to $1$, as described in Section \ref{Labellingcon32}. We will use the following claims to show this result.

    \begin{claim}
        For each $1 \leq i \leq p$, $\sigma(u)$ is odd for each vertex $u \in V_{2i-1}$, and $\sigma(u)$ is even for each vertex $u \in V_{2i}$.
    \end{claim}

    \begin{proof}
        Step $2$ guarantees the result for each vertex $u \neq r$ of degree $k$ (note that, since $i \geq 1$, $u \in V_{2i}$ necessarily implies that $u \neq r$). Moreover, the fact that the edges incident to every vertex of degree $2$ are labelled with consecutive integers in Steps 1 to 3 guarantees the result for each vertex of degree $2$.\qed
    \end{proof}

    This claim justifies that it is impossible to have $\sigma(u) = \sigma(v)$ for some $u$ with degree $2$ and some $v \neq r$ with degree $k$ once the labelling is done.

    \begin{claim}
        For $v_1$ and $v_2$ two vertices with degree $2$, $\sigma(v_1) \neq \sigma(v_2)$.
    \end{claim}

    \begin{proof}
        Obvious from the construction given in Steps 1 to 3, since we label the two edges incident to each vertex of degree $2$ using consecutive values.\qed
    \end{proof}

    We have proven that two vertices with degree $2$ can never have the same value of $\sigma$, at the end of our algorithm. We now aim to prove a similar result for vertices with degree $k$, except $r$.

    \begin{claim}
        For each $1 \leq i \leq p$ and each $u_1,u_2 \in V_{2i}$, we have $\sigma(u_1) \neq \sigma(u_2)$.
    \end{claim}

    \begin{proof}
        Let $u_1, u_2 \in V_{2i}$. Without loss of generality, assume that $\sigma^{\prime}(u_2) \leq \sigma^{\prime}(u_1)$ when Step $2$ begins for $L_i$. The last unlabelled edge incident to $u_1$ (resp. to $u_2$) is labelled by some $\alpha$ (resp. by some $\beta$) during Step $2$ for $L_i$, in which it must be highlighted once again that we sort the vertices of $V_{2i}$ by increasing order of $\sigma^{\prime}$ (which simply implies that $\alpha > \beta$). Then, $\sigma(u_2) < \sigma(u_1)$ follows from $\sigma^{\prime}(u_2) \leq \sigma^{\prime}(u_1)$, since we have $\sigma(u_2) = \sigma^{\prime}(u_2) + \beta \leq \sigma^{\prime}(u_1) + \beta < \sigma^{\prime}(u_1) + \alpha = \sigma(u_1)$.\qed
    \end{proof}

    \begin{claim}\label{claimDegree3}
        For each $1 \leq j < i \leq p$, $u_1 \in V_{2i}$, $u_2 \in V_{2j}$, we have $\sigma(u_1) \neq \sigma(u_2)$.
    \end{claim}

    \begin{proof}
        To begin with, assume that $j = i - 1$. Let $u_1 \in V_{2i}$ and $u_2 \in V_{2i-2}$. Recall that edges in $L_i$ are labelled with $\{a_i,\ldots, b_i\}$. As mentioned in Section \ref{Labellingcon32}, after Step $1$ is done (for the layer $L_i$), we have that $\sigma^{\prime}(u_1) \leq (m_i + M_i - 1)\frac{k-1}{2}$, and once $k-1$ edges incident to $u_2$ have been labelled, $\sigma'(u_2) \geq (m_i + M_i - 1)\frac{k-1}{2}$.

Moreover, at the end of the labelling, we have $\sigma(u_1) \leq b_i + \sigma^{\prime}(u_1)$, since the last unlabelled edge incident to $u_1$ is labelled during the labelling of $L_i$, and $\sigma(u_2) > b_i + \sigma^{\prime}(u_2) \geq \sigma(u_1)$, since the last unlabelled edge incident to $u_2$ is labelled during the labelling of $L_{i-1}$ (as we made sure that the last unpaired edge incident to each such $u_2$ is in $L_{i-1}$ when establishing the lower bound on $\sigma^{\prime}(u_2)$). The same arguments can easily be extended to the case where $j < i-1$, because the value of $\sigma(u_2)$ can only increase in this case.\qed \end{proof}

The proof of the next claim is provided in the Appendix:

    \begin{claim}\label{Claimroot}
        For each vertex $u$ such that $u \neq r$, $\sigma(r) > \sigma(u)$ if $G$ is not $K_{k,2}$.
    \end{claim}

    We have proven, with these claims, that, at the end of the algorithm, all values $\sigma(u)$ for $u \in V$ are different, meaning we do have an antimagic labelling of $G$.\qed \end{proof}

\section{Labelling for non-connected $(k,2)$-bipartite graphs}
\label{Uncon32}

In this section, we will consider that $G = (V,E)$ is a $(k,2)$-bipartite graph, with $k \geq 3$ odd, and $l \geq 2$ connected components $C_1 \sqcup C_2 \sqcup \ldots \sqcup C_l$, each rooted at some vertex $r_i$, $1 \leq i \leq l$, of degree $k$. For every $i$, we define $C_i = (V_i,E_i)$. Let $R = \{r_1, r_2, \ldots, r_l\}$. We assume that $|V_1| \leq |V_2| \leq \ldots \leq |V_l|$. For each $i$ and $j$, we will use $V_i^j$ to denote the set of vertices $u \in V_i$ such that the distance from $r_i$ to $u$ is $j$.

For each connected component, we iteratively apply the labelling described in Section \ref{Labellingcon32}, using the smallest available labels for $C_1$, the next smallest labels for $C_2$, and so on.

We easily obtain that, for any two distinct vertices $u,v \in V$ such that $u,v \notin R$, $\sigma(u) \neq \sigma(v)$. The only issue that could arise is to guarantee that, for every $r_i \in R$, $\sigma(r_i)$ is different from all the others $\sigma(w)$ for $w \in V$.

An important property is that, once this initial labelling is done, we have for every $v_1,v_2$ vertices of degree $2$, $|\sigma(v_1) - \sigma(v_2)| \geq 4$. This allows us to switch some labels, as shown in the Appendix, to prove the following theorem:

\Unconnected*

\section{Antimagic labelling for $(k,2)$-bipartite graphs, with $k \geq 4$ even}

In this section, let $G=(V,E)$ be a $(k,2)$-bipartite graph, with $k \geq 4$ even.

The outline of the labelling process is the same as the one described in Section \ref{Labellingcon32}, except for the labelling of $F_i$ during Step $1$.

Let us consider that we have to label some $F_i$, inside a layer $L_i$, with labels in $[a_i; b_i]$, where $a_i$ is odd and $b_i = a_i + |L_i|-1$. We will first label all edges of $F_i$ that are in $E(V_{2i-1},V_{2i})$, then label the other edges of $F_i$, i.e., the ones that are in $E(V_{2i-2},V_{2i-1})$. We will distinguish two cases:

\begin{itemize}
    \item Case $1$: $|F_i| \mod 4 = 2$. This means that $|F_i|/2$ is odd. Hence, we can simply apply the general idea we described, by labelling all edges in $E(V_{2i-1},V_{2i}) \cap F_i$ with labels arbitrarily chosen in $[a_i; a_i + |F_i|/2 - 1]$. Since $|F_i|/2$ and $a_i$ are odd, the smallest available label after this process will be $a_i + |F_i|/2$, an even number. Moreover, every edge labelled during this step is incident to some vertex $x_j \in V_{2i-1}$ of degree $2$.
    
    We then sort the $x_j$'s by increasing order of $\sigma'$, and we label every unlabelled edge incident to each $x_j$ in this order (they are all in $E(V_{2i-2},V_{2i-1}) \cap F_i$) with the smallest labels available. Note that, since $a_i + |F_i|/2$ is even, $\sigma(x_1)$ is an odd number at the end of the labelling process, and, for every $j \geq 2$, $\sigma(x_{j}) = \sigma(x_{j-1}) + 2$. Once $F_i$ has been labelled, the remainder of the labelling process - Steps $2$ and $3$ - is  the same as the one described in Section \ref{Labellingcon32}.

    For each $u \in V_{2i}$ and each $v \in V_{2i-2}$, we have $\sigma(v) > \sigma(u)$ by construction. Finally, for any two vertices $x_1,x_2 \in V_{2i-1}$ such that both $x_1$ and $x_2$ have their two incident edges in $F_i$, by construction we have either $\sigma(x_1) > \sigma(x_2)$ (if $\sigma'(x_1) > \sigma'(x_2)$) or $\sigma(x_1) < \sigma(x_2)$ (otherwise), and hence $\sigma(x_1) \neq \sigma(x_2)$.

    \item Case $2$: $|F_i| \mod 4 = 0$. In this case, $|F_i|/2$ is even.

    To solve this issue, we will label exactly one additional edge (that will be in $E(V_{2i-2},V_{2i-1}) \cap F_i$) while labelling all the edges of $E(V_{2i-1},V_{2i}) \cap F_i$, in order to make sure that the smallest available label after this step is even.

    More formally, let us first recall that $t_i^k$ is the number of vertices  of type $k$ in $V_{2i}$ (meaning they all have $k$ predecessors and no successor).

    If $t_i^k = 0$, we do the following in Step $1$ of the labelling:

    \begin{itemize}
        \item Take an arbitrary vertex $u \in V_{2i}$ such that $u$ has at least one incident edge $uv$ in $F_i$ ( $F_i$ is empty if such a $u$ does not exist). Label $uv$ with $a_i$, then label the other edge incident to $v$ with $a_i + 1$.
        \item Arbitrarily label all the remaining edges in $F_i \cap E(V_{2i-1},V_{2i})$ with the labels in $[a_i + 2; a_i + \frac{|F_i|}{2}]$.
        \item Label all the unlabelled edges in $F_i \cap E(V_{2i-2},V_{2i-1})$, using the labels in $[a_i + \frac{|F_i|}{2}+1;a_i + |F_i|-1]$ in increasing order, by sorting the vertices in $V_{2i-1}$ by increasing order of $\sigma'$ as in Case $1$.
    \end{itemize}

    Steps $2$ and $3$ are then identical to the ones described in Section \ref{Labellingcon32}.

    Since there are no vertices of type $k$ in $V_{2i}$, every vertex in $V_{2i}$ has at least one edge towards a successor, labelled during the labelling of $L_{i+1}$ with some $\alpha < a_i$. Then, for every vertex $v \in V_{2i-2}$, $v$ has:
    
    \begin{itemize}
        \item $1$ incident edge labelled by some $\alpha' \geq a_i + 1$,
        \item $k-2$ incident edges labelled by some $\beta_j \geq a_i + |F_i|/2 + 1$,
        \item $1$ incident edge in $L_{i-1}$ labelled by some $\gamma > b_i$.
    \end{itemize}
    
    We obtain that, for each $u \in V_{2i}$ and each $v \in V_{2i-2}$, $\sigma(u) < \sigma(v)$ once the labelling (Steps 1 to 3) is done.

    In the following, we will consider that $t_i^k > 0$. Let us arbitrarily order the $q = t_i^k$ vertices of $V_{2i}$ of type $k$: $u_1, u_2, \ldots, u_q$.

    Step $1$ of the labelling will consist of the following substeps (see Figure \ref{Case2illus} for an illustration):

    \begin{enumerate}
        \item Let $vu_1$ be an edge in $F_i$. We label it with $a_i$, then we label the other edge incident to $v \in V_{2i-1}$ with $a_i + 1$.

        \item We then label an edge (in $F_i$) incident to $u_2$, one (in $F_i$) incident to $u_3$, and so on, until one incident edge is labelled for every $u_j$, for $1 \leq j \leq q$. We repeat this step, visiting the $u_j$'s in order, until all $q(k-1)$ edges incident to the $u_j$'s in $F_i$ have been labelled.

        \item We arbitrarily label the unlabelled edges in $E(V_{2i-1},V_{2i}) \cap F_i$ with the smallest labels available.

        \item Label all the unlabelled edges in $E(V_{2i-2},V_{2i-1}) \cap F_i$, using the available labels in increasing order, by sorting the vertices in $V_{2i-1}$ by increasing order of $\sigma'$ as in Case $1$.
    \end{enumerate}

    Steps $2$ and $3$ are then identical to the ones described in Section \ref{Labellingcon32}.

    \begin{figure}[h!]
        \centering
        \begin{tikzpicture}[scale=0.55]
            \node at (0,0) (x1) {$\bullet$};
            \node at (0,-1) (x2) {$\bullet$};
            \node at (0,-2) (x3) {$\bullet$};
            \node at (0,-3) (x4) {$\bullet$};
            \node at (0,-4) (x5) {$\bullet$};
            \node at (0,-5) (x6) {$\bullet$};
            \node at (0,-6) (x7) {$\bullet$};
            \node at (0,-7) (x8) {$\bullet$};
            \node at (2.5,-1.5) (u1) {$\bullet$};
            \node at (2.5,-5.5) (u2) {$\bullet$};

            \draw (u1.center) edge["$1$",sloped] (x1.center);
            \draw (u1.center) edge["$4$",sloped] (x2.center);
            \draw (u1.center) edge["$6$",sloped] (x3.center);
            \draw (u1.center) edge[dashed] (x4.center);
            \draw (u2.center) edge["$3$",sloped] (x5.center);
            \draw (u2.center) edge["$5$",sloped] (x6.center);
            \draw (u2.center) edge["$7$",sloped] (x7.center);
            \draw (u2.center) edge[dashed] (x8.center);
            \draw (x1.center) edge["$2$",sloped] (-2,0);
            \draw (x2.center) edge (-2,-1);
            \draw (x3.center) edge (-2,-2);
            \draw (x4.center) edge[dashed] (-2,-3);
            \draw (x5.center) edge (-2,-4);
            \draw (x6.center) edge (-2,-5);
            \draw (x7.center) edge (-2,-6);
            \draw (x8.center) edge[dashed] (-2,-7);

            \draw[->, line width=1mm] (3,-3.5) -- (4,-3.5);

            \node at (7,0) (x1) {$\bullet$};
            \node at (7,-1) (x2) {$\bullet$};
            \node at (7,-2) (x3) {$\bullet$};
            \node at (7,-3) (x4) {$\bullet$};
            \node at (7,-4) (x5) {$\bullet$};
            \node at (7,-5) (x6) {$\bullet$};
            \node at (7,-6) (x7) {$\bullet$};
            \node at (7,-7) (x8) {$\bullet$};
            \node at (9.5,-1.5) (u1) {$\bullet$};
            \node at (9.5,-5) (u2) {$\bullet$};

            \draw (u1.center) edge["$1$",sloped] (x1.center);
            \draw (u1.center) edge["$4$",sloped] (x2.center);
            \draw (u1.center) edge["$6$",sloped] (x3.center);
            \draw (u1.center) edge[dashed] (x4.center);
            \draw (u2.center) edge["$3$",sloped] (x5.center);
            \draw (u2.center) edge["$5$",sloped] (x6.center);
            \draw (u2.center) edge["$7$",sloped] (x7.center);
            \draw (u2.center) edge[dashed] (x8.center);
            \draw (x1.center) edge["$2$",sloped] (5,0);
            \draw (x2.center) edge["$9$",sloped] (5,-1);
            \draw (x3.center) edge["$11$",sloped] (5,-2);
            \draw (x4.center) edge[dashed] (5,-3);
            \draw (x5.center) edge["$8$",sloped] (5,-4);
            \draw (x6.center) edge["$10$",sloped] (5,-5);
            \draw (x7.center) edge["$12$",sloped] (5,-6);
            \draw (x8.center) edge[dashed] (5,-7);

            \node at (2.5,-8) {$V_{2i}$};
            \node at (9.5,-8) {$V_{2i}$};
            \node at (0,-8) {$V_{2i-1}$};
            \node at (7,-8) {$V_{2i-1}$};

        \end{tikzpicture}
        \caption{Case $|F_i| \mod 4 = 0$, with $k = 4$ and $t_i^k = 2$. In dashed: edges not in $F_i$. It is assumed here that $a_i = 1$ for the sake of clarity.}
        \label{Case2illus}
    \end{figure}

    Once again, we guarantee that, for any two vertices $x_1,x_2 \in V_{2i-1}$, $\sigma(x_1) \neq \sigma(x_2)$, and both are odd. For any two vertices $u_1,u_2 \in V_{2i}$, we also guarantee that $\sigma(u_1) \neq \sigma(u_2)$, and both are even. Moreover, for each vertex $u \in V_{2i}$, we have, once the labelling is done:

    $(i)$ If $u$ is of type $k' < k$, then $u$ has at least one successor, and, as in the case where $t_i^k = 0$, $\sigma(u) < \sigma(v)$ for each $v \in V_{2i-2}$ once the labelling is done.

    $(ii)$ If $u$ is of type $k$:

    $\sigma(u) \leq (a_i + t_i^k) + (a_i + 2t_i^k) + \ldots + (a_i + (k-1)t_i^k) + b_i$

    Indeed, since we visit the $u_j$'s in order during the second step of the labelling of $F_i$ (see above), the difference between two consecutive labels incident to $u$ is exactly $q = t_i^k$.

    Overall, we obtain:

    $\sigma(u) \leq (k-1)a_i + b_i + \frac{(k-1)k}{2}t_i^k$

    We also have, for every vertex $v \neq r \in V_{2i-2}$:

    $\sigma(v) \geq (a_i + 1) + (a_i + \frac{|F_i|}{2} + 1) + (a_i + \frac{|F_i|}{2} + 2) + \ldots + (a_i + \frac{|F_i|}{2} + k - 2) + (b_i + 1)$

    The term $a_i + 1$ comes from the first step described above, as some edge in $E(V_{2i-2},V_{2i-1})$ is labelled with $a_i + 1$. The next smallest label that can be assigned to an edge in $E(V_{2i-2},V_{2i-1})$ is $a_i + \frac{|F_i|}{2} + 1$. Necessarily, $\frac{|F_i|}{2} \geq (k-1)t_i^k$, since every vertex of type $k$ will have $(k-1)$ incident edges in $F_i$, with egality if and only if $V_{2i}$ only contains vertices of type $k$.

    We obtain:

    $\sigma(v) \geq (k-1)a_i + b_i + (k-2)(k-1)t_i^k + \frac{(k-1)(k-2)}{2} + 2$.

    We can thus compare the values of $\sigma(v)$ and $\sigma(u)$:

    $\sigma(v) - \sigma(u) \geq (k-1)t_i^k(k/2 - 2) + \frac{(k-1)(k-2)}{2} + 2$.

    Since $k \geq 4$, $k/2 - 2 \geq 0$, and hence we necessarily have $\sigma(v) > \sigma(u)$.

    We also easily obtain that, for each $u \in V$, $u \neq r$, $\sigma(r) > \sigma(u)$ since $k \geq 4$.

    Since the claims described in Section \ref{Claims} still hold, we have constructed an antimagic labelling for $G$. Note that the connected property is important in this case, as we cannot use the same idea we described in Section \ref{Uncon32}, due to the fact that in this case we do not guarantee that, for any two vertices $x_1,x_2$ of degree $2$, we have $|\sigma(x_1) - \sigma(x_2)| \geq 4$. This yields:
    
\end{itemize}

\ConnectedkEven*

\bibliographystyle{splncs04}
\bibliography{biblio}

\begin{thebibliography}{10}
\providecommand{\url}[1]{\texttt{#1}}
\providecommand{\urlprefix}{URL }
\providecommand{\doi}[1]{https://doi.org/#1}

\bibitem{alon2004}
Alon, N., Kaplan, G., Lev, A., Roditty, Y., Yuster, R.: Dense graphs are
  antimagic. Journal of Graph Theory  \textbf{47},  297--309 (2004).
  \doi{10.1002/jgt.20027}

\bibitem{Berczi2015}
B{\'{e}}rczi, K., Bern{\'{a}}th, A., Vizer, M.: Regular graphs are antimagic.
  Electronic Journal of Combinatorics  \textbf{22}(3), ~3 (2015).
  \doi{10.37236/5465}

\bibitem{chang2015}
Chang, F., Liang, Y.C., Pan, Z., Zhu, X.: Antimagic labeling of regular graphs.
  Journal of Graph Theory  \textbf{82}(4),  339--349 (2016).
  \doi{10.1002/jgt.21905}

\bibitem{cranston2015}
Cranston, D.W.: Regular graphs of odd degree are antimagic. Journal of Graph
  Theory  \textbf{80},  28--33 (2015). \doi{10.1002/jgt.21836}

\bibitem{deng2019}
Deng, K., Li, Y.: Caterpillars with maximum degree 3 are antimagic. Discrete
  Mathematics  \textbf{342},  1799--1801 (2019)

\bibitem{deng2022}
Deng, K., Li, Y.: Antimagic labeling of some biregular bipartite graphs.
  Discussiones Mathematicae Graph Theory  \textbf{42}(4),  1205--1218 (2022)

\bibitem{survey}
Gallian, J.: A dynamic survey of graph labeling. Electronic Journal of
  Combinatorics  \textbf{19} (12 2022). \doi{10.37236/11668}

\bibitem{hartsfield1990}
Hartsfield, N., Ringel, G.: Pearls in Graph Theory: A Comprehensive
  Introduction. Dover Books on Mathematics, Dover Publications (1990)

\bibitem{kaplan09}
Kaplan, G., Lev, A., Roditty, Y.: On zero-sum partitions and anti-magic trees.
  Discrete Mathematics  \textbf{309},  2010--2014 (2009)

\bibitem{wei2025}
Li, W.T.: Antimagic labeling for subdivisions of graphs. Discrete Applied
  Mathematics  \textbf{363},  215--223 (2025). \doi{10.1016/j.dam.2024.12.028}

\bibitem{liang14DM}
Liang, Y.C., Wong, T.L., Zhu, X.: Anti-magic labeling of trees. Discrete
  Mathematics  \textbf{331},  9--14 (2014). \doi{10.1016/j.disc.2014.04.021}

\bibitem{liang14JGT}
Liang, Y.C., Zhu, X.: Antimagic labeling of cubic graphs. Journal of Graph
  Theory  \textbf{75}(1),  31--36 (2014). \doi{10.1002/jgt.21718}

\bibitem{west}
West, D.B.: Introduction to Graph Theory. Prentice-Hall (1996)

\bibitem{Yu2023}
Yu, X.: Antimagic labeling of biregular bipartite graphs. Discrete Applied
  Mathematics  \textbf{327},  47--59 (2023). \doi{10.1016/j.dam.2022.11.019}

\end{thebibliography}

\section*{Appendix}

  \begin{claim}
        For each vertex $u$ such that $u \neq r$, $\sigma(r) > \sigma(u)$ if $G$ is not $K_{k,2}$.
\end{claim}

\begin{proof}

    Recall that $d_G(r) = k$ with $k \geq 3$ odd. Let $u$ be the vertex in $V_2$ with the maximal value of $\sigma$. Let $\lambda_1 < \lambda_2 < \ldots \lambda_j$ be the labels of the edges $uv$, with $v \in V_3$. Let $\lambda_{j+1} < \lambda_{j+2} < \ldots < \lambda_k$ be the labels of the edges $uv$, with $v \in V_1$.

    Since, whenever we label an edge incident to a vertex $v \in V_1$, we immediately label the second edge incident to $v$ using two consecutive labels, we know that $\lambda_{j+1} + 1, \lambda_{j+2} + 1, \ldots, \lambda_{k-1} + 1$ are labels over edges incident to $r$. Moreover, we also know that either $\lambda_k + 1$ or $\lambda_k - 1$ is the label of an other edge incident to $r$, since the choice between the two is made depending on the parity of $\sigma'(u)$ when labelling the last unlabelled edge incident to $u$.

    Let us denote by $\mu_1, \mu_2, \ldots, \mu_j$ the $j$ other labels of $r$. For $1 \leq i \leq j$, the edge labelled by $\lambda_i$ is in $L_2$, and the one labelled by $\mu_i$ is in $L_1$: hence, $\mu_i > \lambda_i$.

    Overall, we obtain that $\sigma(r) \geq \sigma(u) + (k-1) - 1 \geq \sigma(u) + (k-2) > \sigma(u)$ since $k \geq 3$. Moreover, since for each $v \in V_{2j}$ such that $j > 1$, we have $\sigma(v) < \sigma(u)$, we obtain $\sigma(v) < \sigma(r)$.
    
    Finally, for each vertex $v$ of degree $2$, we have $\sigma(v) \leq 2m-1$, so $\sigma(r) \geq (m-1)+(m-3)+(m-5)=3m-9 > 2m-1 \geq \sigma(v)$, since $m \geq 3k \geq 9$. \qed
    
        \end{proof}

\lemmabic*
\begin{proof}

    We claim that the following labelling of $K_{k,2}$ is antimagic:

    \begin{center}
    
    \begin{tikzpicture}
    \node at (0,0) (r) {$\bullet$};
    \node at (3,3) (v1) {$\bullet$};
    \node at (3,1.5) (v2) {$\bullet$};
    \node at (3,0) {$\vdots$};
    \node at (3,-1.5) (vk) {$\bullet$};
    \node at (6,0) (u) {$\bullet$};
    \draw (r.center) -- (v1.center) -- (u.center) -- (v2.center) -- (r.center) -- (vk.center) -- (u.center);

    \node at (0,0.3) {$r$};
    \node at (3,3.3) {$v_1$};
    \node at (3,1.8) {$v_2$};
    \node at (3,-1.8) {$v_k$};
    \node at (6,0.3) {$u$};

    \node at (1.5,1.8) {$2$};
    \node at (1.7,0.6) {$4$};
    \node at (1.5,-1.1) {$2k$};
    \node at (4.5,1.8) {$1$};
    \node at (4.3,0.6) {$3$};
    \node at (4.5,-1.1) {$2k-1$};
    \end{tikzpicture}
    \end{center}

    Indeed, for each $1 \leq i \leq k$, $\sigma(v_i) = 4i - 1$, and hence $\sigma(v_i) \equiv (3 \mod 4)$. Moreover, $\sigma(r) = k(k+1)$ and $\sigma(u) = k^2 \neq \sigma(r)$. Since we have $k^2 \equiv (0 \mod 4)$ if $k$ is even, and $k^2 \equiv (1 \mod 4)$ if $k$ is odd, it necessarily implies that $\sigma(u) \neq \sigma(v_i)$ for each $i$. Finally, if $k \geq 3$ then $\sigma(r) = k^2 + k \geq 4k > \sigma(v_i)$ for each $i$. \qed
\end{proof}

\Unconnected*

\begin{proof}

We will note $[A_i,B_i]$ the set of labels for the edges in $C_i$. One can immediately notice that:

\begin{itemize}
    \item $A_1 = 1$, $B_1 = |E_1|$.
    \item For any $1 \leq i \leq l - 1$, $A_{i+1} = B_i + 1$ and $B_{i+1} = B_i + |E_{i+1}| = |E_1| + |E_2| + \ldots + |E_{i+1}|$.
\end{itemize}

Let $i$ be the smallest index such that there exists a vertex $x \in V \setminus R$ such that $\sigma(r_i) = \sigma(x)$ (if $i$ does not exist, then the labelling is antimagic).

Necessarily, $x$ is a vertex of degree $2$, and $\sigma(r_i)$ is odd. To obtain an antimagic labelling of $G$, the key idea will be to switch up some labels in order to increase $\sigma(r_i)$ by $2$, while keeping the set of values $\Sigma = \{\sigma(v) | v\in V \text{ such that } d_G(v) = 2\}$ unchanged, guaranteeing that $\sigma(r_i)$ cannot be equal to $\sigma(v)$ anymore for some vertex $v$ of degree $2$, thanks to the fact that, after the initial labelling is done, for any two vertices $v_1,v_2$ of degree $2$, we have $|\sigma(v_1) - \sigma(v_2)| \geq 4$.

Let $p$ be the maximum distance from $r_{i+1}$ in $C_{i+1}$. Let $y$ be the vertex in $V^2_i$ such that $\sigma(y)$ is maximal among the vertices in $V^2_i$. Since the vertices are sorted by increasing order of $\sigma'$ in Step $3$ in the labelling described in Section \ref{Labellingcon32}, this means that $y$ is incident to the edge with the largest label in $E(V_i^1,V_i^2)$, meaning $y$ is incident to an edge labelled by either $B_i$ or $B_i-1$. Let $w$ be the vertex in $V^p_{i+1}$ such that $w$ is incident to the edge labelled with the smallest label in $E(V_{i+1}^{p-1},V_{i+1}^p)$. Note that, since the vertices in $V_{i+1}^p$ cannot have any successors, Steps $1$ and $3$ imply that, actually, $w$ is necessarily incident to the edge labelled with the smallest label used for $C_{i+1}$, i.e., with $B_i + 1 = A_{i+1}$.

We will distinguish two cases, depending on if the vertices in $V^p_{i+1}$ (and hence also $w$) have degree $2$ or $k$.

\begin{itemize}
    \item Case $1$: $w$ has degree $2$.

    In this case, $w$ is necessarily incident to the edge labelled with $B_i + 2$, and $w$ is the vertex with the minimal value of $\sigma$ among the vertices in $V_{i+1}^p$. Let $u_1$ and $u_2$ be the two neighbors of $w$ (see Figure \ref{32case1} for an illustration):

    \begin{figure}[h!]
        \centering
        \begin{tikzpicture}[scale=0.7]
            \node at (0,0) (r) {$\bullet$};
            \node at (2,2) (u1) {$\bullet$};
            \node at (2,0) (u2) {$\bullet$};
            \node at (2,-2) (u3) {$\bullet$};
            \node at (4,-2) (u4) {$\bullet$};
            \node at (4,-5) (w) {$\bullet$};
            \node at (2,-4) (v1) {$\bullet$};
            \node at (2,-6) (v2) {$\bullet$};

            \draw (r.center) edge (u1.center);
            \draw (r.center) edge (u2.center);
            \draw (r.center) edge["$B_i$",sloped] (u3.center);
            \draw (u3.center) edge["$B_i-1$"] (u4.center);
            \draw[dashed] (u1.center) edge (4,2);
            \draw[dashed] (u2.center) edge (4,0);
            \draw (v1.center) edge["$B_i+1$",sloped] (w.center);
            \draw (w.center) edge["$B_i+2$",sloped] (v2.center);
            \draw[dashed] (v1.center) edge (0,-3.5);
            \draw[dashed] (v1.center) edge (0,-4.5);
            \draw[dashed] (v2.center) edge (0,-5.5);
            \draw[dashed] (v2.center) edge (0,-6.5);

            \node at (-0.2,0.2) {$r_i$};
            \node at (4.2,-4.8) {$w$};
            \node at (4.2,-2.2) {$y$};

            \node at (2,-3.7) {$u_1$};
            \node at (2,-6.3) {$u_2$};

            \draw[->, line width=1mm] (5,-2.5) -- (6,-2.5);

            \node at (7,0) (r) {$\bullet$};
            \node at (9,2) (u1) {$\bullet$};
            \node at (9,0) (u2) {$\bullet$};
            \node at (9,-2) (u3) {$\bullet$};
            \node at (11,-2) (u4) {$\bullet$};
            \node at (11,-5) (w) {$\bullet$};
            \node at (9,-4) (v1) {$\bullet$};
            \node at (9,-6) (v2) {$\bullet$};

            \draw (r.center) edge (u1.center);
            \draw (r.center) edge (u2.center);
            \draw (r.center) edge["$B_i+2$",sloped] (u3.center);
            \draw (u3.center) edge["$B_i+1$"] (u4.center);
            \draw[dashed] (u1.center) edge (11,2);
            \draw[dashed] (u2.center) edge (11,0);
            \draw (v1.center) edge["$B_i-1$",sloped] (w.center);
            \draw (w.center) edge["$B_i$",sloped] (v2.center);
            \draw[dashed] (v1.center) edge (7,-3.5);
            \draw[dashed] (v1.center) edge (7,-4.5);
            \draw[dashed] (v2.center) edge (7,-5.5);
            \draw[dashed] (v2.center) edge (7,-6.5);

            \node at (6.8,0.2) {$r_i$};
            \node at (11.2,-4.8) {$w$};
            \node at (11.2,-2.2) {$y$};

            \node[color=blue] at (6.8,-0.4) {$+2$};
            \node[color=blue] at (11.2,-2.6) {$+2$};
            \node[color=red] at (9,-2.4) {$+4$};
            \node[color=red] at (11.1,-5.3) {$-4$};
            \node[color=purple] at (9,-4.4) {?};
            \node[color=purple] at (9,-6.8) {?};

            \node at (9,-3.7) {$u_1$};
            \node at (9,-6.3) {$u_2$};

            \node at (-2,0) {$C_i$};
            \node at (-2,-5) {$C_{i+1}$};
        \end{tikzpicture}
        \caption{Case $1$, $k = 3$}
        \label{32case1}
    \end{figure}

    We can then switch the pair of labels $(B_i, B_i - 1)$ with $(B_i + 2, B_i + 1)$, and ``regenerate'' the labelling for $C_{i+1}$ (using the labels in $\{B_i + 3, B_i + 4, \ldots, B_{i+1}\}$), using the same process as the one described in Section \ref{Labellingcon32}. We need to regenerate the labelling because we have changed the values of $\sigma(u_1)$ and $\sigma(u_2)$. However, during the labelling process, this only changes the order in which vertices are treated during Step $2$, and hence by regenerating the labelling in $C_{i+1}$ we guarantee that the claims shown in Section \ref{Claims} still hold.

    Notice that, since we are switching two pairs of consecutive labels, the set $\Sigma$ is unchanged, and, for any two vertices $v_1,v_2$  of degree $2$, $|\sigma(v_1) - \sigma(v_2)| \geq 4$. The only meaningful changes are $\sigma(r_i)$, $\sigma(y)$, $\sigma(u_1)$ and $\sigma(u_2)$. Specifically, $\sigma(r_i)$ is increased by $2$, guaranteeing it cannot be equal to any other $\sigma(z)$, for any vertex $z$. Moreover, $\sigma(y)$ is also increased by $2$, which does not matter since $y$ was the vertex with the highest value of $\sigma$ among the vertices in $V^2_i$, and $\sigma(y)$ remains even. Notice that, although we have changed $\sigma(u_1)$ and $\sigma(u_2)$, we still necessarily have $\sigma(r_i) < \sigma(u_1)$ and $\sigma(r_i) < \sigma(u_2)$.

    \item Case $2$ : $w$ has degree $k$ (see Figure \ref{32case2} for an illustration):

    \begin{figure}[h!]
        \centering
        \begin{tikzpicture}[scale=0.7]
            \node at (0,0) (r) {$\bullet$};
            \node at (2,2) (u1) {$\bullet$};
            \node at (2,0) (u2) {$\bullet$};
            \node at (2,-2) (u3) {$\bullet$};
            \node at (4,-2) (u4) {$\bullet$};
            \node at (4,-5) (w) {$\bullet$};
            \node at (2,-3.5) (v1) {$\bullet$};
            \node at (2,-5) (v2) {$\bullet$};
            \node at (2,-6.5) (v3) {$\bullet$};
            \node at (0,-3.5) (x1) {$\bullet$};

            \draw (r.center) edge (u1.center);
            \draw (r.center) edge (u2.center);
            \draw (r.center) edge["$B_i$",sloped] (u3.center);
            \draw (u3.center) edge["$B_i-1$"] (u4.center);
            \draw[dashed] (u1.center) edge (4,2);
            \draw[dashed] (u2.center) edge (4,0);
            \draw (v1.center) edge["$B_i+1$",sloped] (w.center);
            \draw (w.center) edge (v2.center);
            \draw (w.center) edge (v3.center);
            \draw (x1.center) edge["$B_i + 2$"] (v1.center);
            \draw[dashed] (v2.center) edge (0,-5);
            \draw[dashed] (v3.center) edge (0,-6.5);

            \node at (-0.2,0.2) {$r_i$};
            \node at (4.2,-4.8) {$w$};
            \node at (4.2,-2.2) {$y$};

            \node at (2.2,-3.2) {$v_1$};
            \node at (2.2,-4.7) {$v_2$};
            \node at (2.2,-6.8) {$v_3$};
            \node at (-0.2,-3.2) {$u$};

            \draw[->, line width=1mm] (5,-2.5) -- (6,-2.5);

            \node at (7,0) (r) {$\bullet$};
            \node at (9,2) (u1) {$\bullet$};
            \node at (9,0) (u2) {$\bullet$};
            \node at (9,-2) (u3) {$\bullet$};
            \node at (11,-2) (u4) {$\bullet$};
            \node at (11,-5) (w) {$\bullet$};
            \node at (9,-3.5) (v1) {$\bullet$};
            \node at (9,-5) (v2) {$\bullet$};
            \node at (9,-6.5) (v3) {$\bullet$};
            \node at (7,-3.5) (x1) {$\bullet$};

            \draw (r.center) edge (u1.center);
            \draw (r.center) edge (u2.center);
            \draw (r.center) edge["$B_i+2$",sloped] (u3.center);
            \draw (u3.center) edge["$B_i+1$"] (u4.center);
            \draw[dashed] (u1.center) edge (11,2);
            \draw[dashed] (u2.center) edge (11,0);
            \draw (v1.center) edge["$B_i-1$",sloped] (w.center);
            \draw (w.center) edge (v2.center);
            \draw (w.center) edge (v3.center);
            \draw (x1.center) edge["$B_i$"] (v1.center);
            \draw[dashed] (v2.center) edge (7,-5);
            \draw[dashed] (v3.center) edge (7,-6.5);

            \node at (6.8,0.2) {$r_i$};
            \node at (11.2,-4.8) {$w$};
            \node at (11.2,-2.2) {$y$};
            \node at (9.2,-3.2) {$v_1$};
            \node at (9.2,-4.7) {$v_2$};
            \node at (9.2,-6.8) {$v_3$};
            \node at (6.8,-3.2) {$u$};

            \node at (-2,0) {$C_i$};
            \node at (-2,-5) {$C_{i+1}$};
        \end{tikzpicture}
        \caption{Case $2$, $k = 3$}
        \label{32case2}
    \end{figure}

    Again, we want to switch the pair of labels $(B_i, B_i - 1)$ and $(B_i + 1, B_i + 2)$. As in the previous case, this guarantees that $\sigma(r_i)$ cannot be in conflict with any other vertex, and increases $\sigma(y)$ by $ 2$, which does not create any additional issue.

    However, we have to be more careful when regenerating the labelling for $C_{i+1}$, as guaranteeing good upper bounds (for vertices in $V_{i+1}^p$) and lower bounds (for vertices in $V_{i+1}^{p-2}$) is slightly harder than in Case $1$.

    Notice that, if $C_{i+1} = K_{k,2}$ (and hence $C_{i} = K_{k,2}$ also), then we have an explicit knowledge of the values of $\sigma$, using Lemma \ref{lemmabic}: $\sigma(r_i) = k(k+1) + 2(i-1)k^2$, and $\sigma(y) = k^2 + 2ik^2 = \sigma(r_i) + 2k^2 - k$. We can then simply switch the pair of labels $(B_i + 1, B_i + 2)$ with $(B_i, B_i - 1)$, without creating any additional issues. We will thus consider in the following that $C_{i+1} \neq K_{k,2}$.

    In the following, for the sake of simplicity, we will denote by $F$ the set of edges that we label during Step 1 (see Section \ref{Labellingcon32}) for the layer of $C_{i+1}$ that consists of the edges of $G[V_{i+1}^{p-2}\cup V_{i+1}^{p-1} \cup V_{i+1}^p]$.

   Let us explain how we construct $F$ in this case.

    First, notice that every vertex in $V_{i+1}^p$ is of type $k$, which means $F$ will be made of $(k-1)$ edges incident to every vertex in $V_{i+1}^p$ (along with the corresponding edges in $E(V_{i+1}^{p-2},V_{i+1}^{p-1})$). We distinguish two cases:

    \begin{itemize}
        \item{Case 2.1:} There exists at least one vertex $u \in V_{i+1}^{p-2}$ of type $1 < t < k$, meaning $u$ has at least $2$ predecessors and at least $1$ successor. Recall that $C_{i+1} \neq K_{k,2}$, meaning every vertex in $V_{i+1}^{p-2}$ has at least one predecessor.

        In this case, let us define $v \in V_{i+1}^{p-1}$ a successor of $u$, and $w'$ the neighbor of $v$ in $V_{i+1}^p$. There exists a path $u \rightarrow v \rightarrow w'$, and we put $uv$ and $vw'$ in $F$. We then greedily complete $F$ such that the properties described in Step $1$ from Section \ref{Labellingcon32} still hold. In particular, since $w'$ has $k-1 \geq 2$ incident edges in $F$, $w'$ has a neighbour $v' \in V_{i+1}^{p-1}$, different from $v$, and $v'$ has a neighbour $u' \in V_{i+1}^{p-2}$ (which may be equal to $u$), such that $w'v' \in F$ and $u'v' \in F$. We then start the labelling of $F$ on $u'$ (instead of starting from an arbitrary vertex in $V_{i+1}^{p-2}$), with $m_i = B_i + 1$ and $M_i = B_i + 2(k-1)|V_{i+1}^p|$, by labelling the edge $u'v'$ by $M_i$ and the edge $v'w'$ by $M_i-1$. Then, we label the edge $vw'$ by $m_i=B_i+1$ (instead of labelling by $m_i$ an arbitrary edge incident to $w'$ in $F$) and the edge $uv$ by $m_i+1=B_i+2$, and finally we complete the labelling of $F$ by following Step $1$ from Section \ref{Labellingcon32}.

        We then switch the pair of labels $(B_i + 1, B_i + 2)$ and $(B_i - 1, B_i)$: since this decreases the value of $\sigma'(w')$, this means that the upper bound of $(m_i + M_i - 1)\frac{k-1}{2}$ will still be verified. This also decreases the value of $\sigma'(u)$: however, since $u$ has at least two predecessors, we can pair the label $B_i = m_i-1$ with the label of one of these two edges towards predecessors (which is greater than $M_i$), and obtain a sum at least equal to $m_i + M_i - 1$, so that $\sigma'(u) \geq (m_i + M_i - 1)\frac{k-1}{2}$ still holds after $k-1$ edges incident to $u$ have been labelled. Moreover, $\sigma(u)$ is obtained by adding the label of the last edge incident to $u$, that is also incident to a predecessor of $u$.

        Overall, we have guaranteed the necessary lower and upper bounds, and can proceed with the re-labelling of $C_{i+1}$, without any additional issues.

        \item{Case 2.2:} Every vertex in $V_{i+1}^{p-2}$ is either of type $1$ or $k$. As vertices of type $k$ have no edges towards any vertex in $V_{i+1}^{p-1}$, there must exist at least one vertex of type $1$ in $V_{i+1}^{p-2}$, and we will only consider such vertices.

        Let $u$ be a vertex of $V_{i+1}^{p-2}$ such that $u$ is of type $1$ and has at least $2$ distinct vertices in its second neighborhood in $V_{i+1}^{p}$. Such a vertex $u$ does exist, as each vertex in $V_{i+1}^{p-2}$ has degree $0$ or $k-1$ in the layer consisting of the edges of $G[V_{i+1}^{p-2}\cup V_{i+1}^{p-1} \cup V_{i+1}^p]$, while each vertex in $V_{i+1}^p$ has degree $k$ in this layer.
        
       Let $w_1,w_2 \in V_{i+1}^{p}$ be two distinct vertices in the second neighborhood of $u$. Hence, there exist two paths $u \rightarrow v_1 \rightarrow w_1$ and $u \rightarrow v_2 \rightarrow w_2$. We add the edges $uv_1$ and $v_1w_1$ to $F$, and then we complete $F$ in such a way that $v_2w_2$ is the edge incident to $w_2$ that is not in $F$.

        Let $z$ be a vertex in $V_{i+1}^{p-2}$ such that there exists a path $z \rightarrow v_3 \rightarrow w_1$ and $zv_3$ and $v_3w_1$ are in $F$, with $v_3 \neq v_1$. Note that $z$ can be equal to $u$, and $z$ exists since $w_1$ has $k-1 \geq 2$ incident edges in $F$. We then apply the algorithm described in Section \ref{Labellingcon32} to label $F$: we start the labelling of $F$ from $z$, and thus we label the edge $zv_3$ by $M_i = B_i + 2(k-1)|V_{i+1}^p|$, and the edge $v_3w_1$ by $M_i - 1$. We then label the path $u \leftarrow v_1 \leftarrow w_1$, the edge $w_1v_1$ being labelled by $m_i = B_i + 1$, and the edge $uv_1$ by $m_i + 1 = B_i + 2$. We then continue as described in Step 1 of Section \ref{Labellingcon32}, until all the edges of $F$ have been labelled. Recall that each vertex of $V_{i+1}^p$ is of type $k$, meaning each one has $k-1$ incident edges in $F$. Since $k-1$ is even, every time we reach a vertex in $V_{i+1}^p$ during the labelling of $F$, we will be able to extend the current path and immediately label another edge incident to this vertex. This means that, once $F$ has been labelled, all the vertices in $V_{i+1}^p$ will have the same value of $\sigma'$ before starting Step $2$, equal to $(m_i + M_i - 1) \frac{k-1}{2}$.
        
        Again, we switch the pair of labels $(B_i+1,B_i + 2)$ and $(B_i-1,B_i)$. This decreases $\sigma'(w_1)$ by $2$, which does not create any additional issue. This also means that $w_1$ has the smallest value of $\sigma'$ among all the vertices in $V_{i+1}^p$, and hence that its last unlabelled incident edge $w_1v'_1 \not \in F$ will be labelled first during Step $2$, from the way this step works. Therefore, the edges $v'_1w_1$ and $u'_1v'_1$ (where $u'_1$ is the neighbour of $v'_1 \in V_{i+1}^{p-1}$ in $V_{i+1}^{p-2}$) are labelled by $M_i + 1$ and $M_i + 2$.
        
        This switch of labels also decreases $\sigma'(u)$ by $2$. Recall that $w_2v_2 \not \in F$ is the last unlabelled edge incident to $w_2$, and that the labels $M_i + 1$ and $M_i + 2$ have already been used in Step 2, meaning that, in Step 2, the edge $w_2v_2$ will be labelled by some $\beta' \geq M_i+3$, and similarly the edge $uv_2$ will be labelled by some $\beta \geq M_i + 3$. Also recall that, in Section \ref{Labellingcon32}, we used the fact that $\beta \geq M_i + 1$, both when we were in Case $1b$ and $u$ had an odd number of incident edges in $F$, but also for each such label $\beta \in K_2$, in order to obtain the lower bound of $\sigma'(u) \geq (m_i + M_i - 1) \frac{k-1}{2}$ after $k-1$ edges incident to $u$ have been labelled (the $k$th and last labelled edge incident to $u$ being also incident to a predecessor of $u$). Since we decreased $\sigma'(u)$ by $2$, but we have $\beta \geq M_i + 3$ instead of $\beta \geq M_i + 1$, we can still guarantee the same lower bound on $\sigma'(u)$.

        We can then again proceed with the regeneration of the labelling in $C_{i+1}$ without any additional issues.
        
    \end{itemize}

\end{itemize}

We can now repeat this process until there are no more $i$ such that $\sigma(r_i)$ is equal to $\sigma(x)$ for some other vertex $x$ of the graph, meaning the labelling obtained at the end is antimagic.\qed
\end{proof}

\end{document}